\documentclass[a4paper,10pt,reqno]{amsart}

\usepackage[utf8]{inputenc}
\usepackage[T1]{fontenc}
\usepackage{amsthm}
\usepackage{amsmath}
\usepackage{amssymb}
\usepackage[inline]{enumitem}
\usepackage{comment}
\PassOptionsToPackage{hyphens}{url}\usepackage{hyperref}
\usepackage{fancyhdr}
\usepackage{mathrsfs}
\usepackage{stmaryrd}
\usepackage{tikz-cd}
\usetikzlibrary{arrows}
\usepackage[usestackEOL]{stackengine}
\usepackage{scalerel}

\theoremstyle{definition}

\newtheorem{theorem}{Theorem}[section]
\newtheorem{lemma}[theorem]{Lemma}

\newtheorem{corollary}[theorem]{Corollary}

\theoremstyle{definition}
\newtheorem{definition}[theorem]{Definition}

\newtheorem{examples}[theorem]{Examples}
\newtheorem{remarks}[theorem]{Remarks}

\definecolor{blue-url}{RGB}{0,0,100}
\definecolor{red-url}{RGB}{100,0,0}
\definecolor{green-url}{RGB}{0,100,0}
\definecolor{light-yellow}{RGB}{255,255,128}
\definecolor{light-blue}{RGB}{193,255,255}
\definecolor{light-red}{RGB}{239,83,80}

\hypersetup{
	pdftitle={Atomicity},
	pdfauthor={Tringali},
	pdfmenubar=false,
	pdffitwindow=true,
	pdfstartview=FitH,
	colorlinks=true,
	linkcolor=blue-url,
	citecolor=green-url,
	urlcolor=red-url
}

\renewcommand{\emptyset}{\varnothing}
\renewcommand{\setminus}{\smallsetminus}
\renewcommand{\,}{\kern 0.1em}

\providecommand\llb{\llbracket}
\providecommand\rrb{\rrbracket}

\newcommand{\evid}[1]{\textsf{#1}}

\def\myoverline#1{\ThisStyle{%
  \setbox0=\hbox{$\SavedStyle#1$}%
  \stackengine{1.7pt}{$\SavedStyle#1$}{\rule{\wd0}{0.05pt}}{O}{c}{F}{F}{S}%
}}

\def\myoverlinesub#1{\ThisStyle{%
  \setbox0=\hbox{$\SavedStyle#1$}%
  \stackengine{2.1pt}{$\SavedStyle#1$}{\rule{\wd0}{0.1pt}}{O}{c}{F}{F}{S}%
}}

\makeatletter
\DeclareFontFamily{OMX}{MnSymbolE}{}
\DeclareSymbolFont{MnLargeSymbols}{OMX}{MnSymbolE}{m}{n}
\SetSymbolFont{MnLargeSymbols}{bold}{OMX}{MnSymbolE}{b}{n}
\DeclareFontShape{OMX}{MnSymbolE}{m}{n}{
	<-6>  MnSymbolE5
	<6-7>  MnSymbolE6
	<7-8>  MnSymbolE7
	<8-9>  MnSymbolE8
	<9-10> MnSymbolE9
	<10-12> MnSymbolE10
	<12->   MnSymbolE12
}{}
\DeclareFontShape{OMX}{MnSymbolE}{b}{n}{
	<-6>  MnSymbolE-Bold5
	<6-7>  MnSymbolE-Bold6
	<7-8>  MnSymbolE-Bold7
	<8-9>  MnSymbolE-Bold8
	<9-10> MnSymbolE-Bold9
	<10-12> MnSymbolE-Bold10
	<12->   MnSymbolE-Bold12
}{}

\let\llangle\@undefined
\let\rrangle\@undefined
\DeclareMathDelimiter{\llangle}{\mathopen}%
{MnLargeSymbols}{'164}{MnLargeSymbols}{'164}
\DeclareMathDelimiter{\rrangle}{\mathclose}%
{MnLargeSymbols}{'171}{MnLargeSymbols}{'171}
\makeatother

\begin{document}
\title{A characterization of atomicity}
\author{Salvatore Tringali}
\address{School of Mathematical Sciences,
Hebei Normal University | Shijiazhuang, Hebei province, 050024 China}
\email{salvo.tringali@gmail.com}
\urladdr{http://imsc.uni-graz.at/tringali}
%
\subjclass[2020]{Primary 13A05, 13E99, 16P70, 20M10, 20M13.}
%
%
%
\keywords{Chain conditions, atoms, characterization, factorizations, monoids, preorders.}
%
%
\begin{abstract}
In [Math.~Proc.~Cam\-bridge Philos.~Soc.~\textbf{64} (1968), 251--264], P.\,M.~Cohn famously claimed that a commutative domain is atomic if and only if it satisfies the ascending chain condition on principal ideals (ACCP). Some years later, a counterexample was provided by A.~Grams in [Math.~Proc.~Cambridge Philos.~Soc.~\textbf{75} (1974), 321--329]: Every commutative domain with the ACCP is atomic, but not vice versa. This has led to the problem of finding  a sensible (ideal-theoretic) characterization of atomicity.

The question (explicitly stated on p.~3 of A.~Geroldinger and F.~Halter-Koch's 2006 mon\-o\-graph on fac\-tor\-i\-za\-tion) is still open. We settle it using the language of monoids and preorders.
\end{abstract}
\maketitle
\thispagestyle{empty}

\section{Introduction}
\label{sec:intro}
A (multiplicatively written) monoid $H$ is \evid{cancellative} if the function $H \to H \colon x \mapsto uxv$ is injective for all $u, v \in H$; \evid{unit-cancellative} if $xy \ne x \ne yx$ for all $x, y \in H$ with $y$ not a unit; and \evid{acyclic} if $uxv \ne \allowbreak x$ for all $u, v, x \in H$ unless $u$ and $v$ are both units (we address the reader to J.\,M.~Howie's monograph \cite{Ho95} for generalities on monoids). 

An acyclic or cancellative monoid is unit-cancellative, but not conversely; and it is a basic fact that a cancellative monoid sat\-is\-fy\-ing the ascending chain condition (ACC) on both principal left ideals (ACCPL) and principal right ideals (ACCPR) is \evid{atomic}, namely, each non-unit is a product of atoms (we recall that an \evid{atom}, in an arbitrary monoid, is a non-unit that does not factor as a product of two non-units). We refer to this fundamental result as \emph{Cohn's theorem}, since it can be traced back to P.\,M.~Cohn's work on factorization in the 1960s (e.g., see \cite[Theorem 2.8]{Co64}, the unnumbered corollary on p.~589 of \cite{Co69}, 
and \cite[Proposition 0.9.3]{Co06}).

Cohn's theorem was extended to unit-cancellative monoids in \cite[Theorem 2.28(i)]{Fa-Tr18} and then generalized to premons in \cite[Theorem 3.10]{Tr20(c)} and \cite[Theorem 3.4]{Co-Tr-21(a)}, where a \evid{premon} (or \evid{premonoid}) is a pair consisting of a monoid $H$ and a preorder --- i.e., a reflexive and transitive binary relation --- on (the carrier set of) $H$. A key to these arguments is the role played by the \evid{divisibility preorder} $\mid_H$, viz., the binary relation on $H$ defined by $x \mid_H y$ if and only if $x \in H$ and $y \in HxH$. In fact, the result follows from applying \cite[Theorem 3.10]{Tr20(c)} to the \evid{divisibility premon} $(H, \mid_H)$ of $H$ and considering that, by \cite[Corollary 4.6]{Tr20(c)}, $H$ is unit-cancellative and satisfies the ACCPL and the ACCPR if and only if it is acyclic and satisfies the ACC on principal two-sided ideals (ACCP).

The interplay between ACCs and factorization in \emph{commutative} monoids is a classical topic which has overseen a revival in recent years.
In \cite[Proposition 1.1]{Co68}, Cohn famously claimed (without proof) that a commutative domain $R$ is atomic (i.e., the multiplicative monoid $R^\bullet$ of the non-zero elements of $R$ is atomic) if and only if $R$ satisfies the ACCP (i.e., $R^\bullet$ satisfies the ACCP). Some years later, A.~Grams \cite{Grams74} showed, by way of a counterexample, that Cohn's assertion is wrong. Indeed, every commutative domain with the ACCP is atomic, but not vice versa. Grams' construction is usually acknowledged as the \emph{first} counter\-example, but it seems that Cohn had already realized his own mistake and outlined a simpler construction in \cite[p.~4, lines~14--18]{Co73}.

Further contributions in the same vein were subsequently made by A.~Zaks \cite{Zaks82}, who considered cer\-tain quotients of a polynomial ring in infinitely many variables and proved that they are atomic but do not satisfy the ACCP; and by M.~Roitman, who showed the existence of an atomic commutative domain $R$ such that the univariate polynomial ring $R[X]$ is not atomic \cite[Example 5.1]{Roit93}. Incidentally, Roitman's example produced an atomic commutative domain without the ACCP (if $R$ had the ACCP, then we would gather from \cite[Theorem 14.6]{Gil84} that $R[X]$ also has the ACCP and hence is atomic by Cohn's theorem). More recently, J.\,G.~Boynton and J.~Coykendall \cite{Boy-Coy2019} have used pullbacks of commutative rings to generate large families of atomic commutative domains that do not satisfy the ACCP; F.~Gotti and B.~Li \cite[Theorem 4.4]{GotLi22(a)} have built what appears to be the first example of an atomic, commutative monoid domain without the ACCP; and J.~Bell et al.~\cite[Proposition 7.6]{Be-Br-Na-Sm22} have provided the first example of an atomic, non-commutative, finitely presented monoid domain that satisfies neither the ACCPL nor the ACCPR (see also
\cite{Coy-Go2019} for some related results on monoid rings, atomicity, and the ACCP).

It is definitely easier to come up with cancellative commutative monoids that are lacking the ACCP. E.g., S.\,T.~Chapman et al.~proved in \cite[Corollary 4.4]{Cha-Got-Got21} that, if $r$ is a non-zero rational number smaller than $1$ and the numerator (of the reduced fraction) of $r$ is not $1$, then the submonoid of the ad\-di\-tive group of the rational field generated by $1, r, r^2, \ldots$ is atomic but does not satisfy the ACCP.

With these preliminaries in place, it is natural to ask if Cohn's false claim (that, for commutative domains, atomicity is equivalent to the ACCP) can be fixed by providing a sensible characterization (of an ideal-theoretic nature) of when a cancellative commutative monoid is atomic. In this regard, the last lines of p.~3 in A.~Geroldinger and F.~Halter-Koch's 2006 monograph \cite{GeHK06} on non-u\-nique factorization read, ``Up to now, there is no satisfactory ideal-theoretic characterization of atomic [commutative] domains.''
Geroldinger has confirmed in private communication that, to his knowledge, the problem --- ostensibly belonging to folklore --- is still open.

In the present paper, we aim to settle the question by proving, more generally, a characterization of \emph{factorability} in the abstract setting of premons (Corollary \ref{cor:characterization-of-factoriable}). First, we demonstrate that, in a locally artinian premon $(H, \preceq)$, every $\preceq$-non-unit factors as a product of $\preceq$-irreducibles (Theorem \ref{thm:FTF-with-local-artinianity}). Next, we obtain a char\-ac\-ter\-i\-za\-tion of atomicity (Corollary \ref{cor:characterize-atomicity-in-acyclic-monoids}) by (i) restricting the previous result to the case where $H$ is acyclic and $\preceq$ is the divisibility preorder on $H$, (ii) recognizing that all $\preceq$-irreducibles are then atoms, (iii) reinterpreting the condition of local artinianity in ideal-theoretic terms, and (iv) considering that, among many others, cancellative commutative monoids are acyclic. Details will be given in Sect.~\ref{sec:premonoids} (see, in particular, Definition \ref{def:artinianity-and-the-like}), but something to keep in mind is that we use the adverb ``locally'' to refer to an element-wise property (i.e., the term has nothing to do with prime ideals and localizations in the sense of \cite[Sect.~2.2]{GeHK06}). 

Overall, this work is simple if measured from the technicality of the proofs. Its value, we hope, lies rather in the insight that the ACCP has little to do with the classical setting \cite{Ger-Zho2020} of factorization theory (an observation already made in \cite{Tr20(c)}) and is the first step in a count\-a\-bly infinite ladder of weaker and weaker conditions ultimately ``converging'' to local artinianity (Remarks \ref{rem:artinianity}).

\section{Turning the ACCP to an element-wise condition}
\label{sec:premonoids}

Let $(H, \preceq)$ be a premon (note that, in principle, we require no compatibility between the monoid operation and the preorder).
An el\-e\-ment $u \in H$ is a \evid{$\preceq$-unit} if $u \preceq 1_H \preceq u$ and a \evid{$\preceq$-non-unit} oth\-er\-wise. 
A \evid{$\preceq$-quark} is then a $\preceq$-non-unit $a \in \allowbreak H$ with the property that there is no $\preceq$-non-unit $b \prec a$ (i.e., $b \preceq a$ and $a \not\preceq b$); and given $s \in \mathbb N_{\ge 2} \cup \{\infty\}$, a \evid{$\preceq$-ir\-re\-duc\-i\-ble of degree $s$} (or \evid{degree-$s$ $\preceq$-ir\-red\-u\-ci\-ble}) is a $\preceq$-non-unit $a$ such that $a \ne x_1 \cdots x_k$ for every $k \in \llb 2, s \rrb$ and all $\preceq$-non-units $x_1 \prec a, \ldots, x_k \prec a$. In par\-tic\-u\-lar, we refer to a $\preceq$-ir\-re\-duc\-i\-ble of degree $2$ as a \evid{$\preceq$-ir\-re\-duc\-i\-ble} (occasionally, the term may also be used as an adjective).

The \evid{$\preceq$-height} of an element $x \in H$ is, on the other hand, the supremum of the set of all $n \in \mathbb N^+$ for which there are $\preceq$-non-units $x_1, \ldots, x_n$ with $x_1 = x$ and $x_{i+1} \prec x_i$ for each $i \in \llb 1, n-1 \rrb$, where $\sup \emptyset := \allowbreak 0$. Of course, $x$ is a $\preceq$-unit if and only if its $\preceq$-height is zero; and is a $\preceq$-quark if and only if its $\preceq$-height is one (in general, one cannot say much about the $\preceq$-height of a $\preceq$-irreducible).

The no\-tions of $\preceq$-[non-]unit, $\preceq$-quark, $\preceq$-ir\-red\-u\-ci\-ble, and $\preceq$-height were introduced in \cite[Definitions 3.6 and 3.11]{Tr20(c)},
whereas $\preceq$-ir\-re\-duc\-i\-bles of \emph{finite} degree were first considered in \cite[Definition 3.1]{Co-Tr-21(a)}. Note that a $\preceq$-quark is $\preceq$-ir\-red\-u\-ci\-ble, but the converse need not be true \cite[Remark 3.7(4)]{Tr20(c)}.

\begin{definition}
\label{def:artinianity-and-the-like}
\begin{enumerate*}[label=\textup{(\arabic{*})}, mode=unboxed]
\item\label{def:artinianity-and-the-like(1)}
Given a premon $(H, \preceq)$, an element $x \in H$ is \evid{$\preceq$-artinian} if there exists no (strictly) $\preceq$-decreasing sequence $x_1, x_2, \ldots$ in $H$ with $x_1 = x$, and \evid{strongly $\preceq$-artinian} if the $\preceq$-height of $x$ is finite. 
\end{enumerate*}

\vskip 0.05cm

\begin{enumerate*}[label=\textup{(\arabic{*})}, mode=unboxed, resume]
\item\label{def:artinianity-and-the-like(2)}
The premon itself is then \evid{artinian} (resp., \evid{strongly artinian}) if every $\preceq$-non-unit is $\preceq$-artinian (resp., strongly $\preceq$-artinian); and \evid{$k$-locally} (resp., \evid{strongly $k$-locally}) \evid{artinian}, for a certain $k \in \allowbreak \mathbb N^+ \cup \allowbreak \{\infty\}$, if every $\preceq$-non-unit is a product of $k$ or fewer $\preceq$-artinian (resp., strongly $\preceq$-artinian) $\preceq$-non-units.
\end{enumerate*}

\vskip 0.05cm

\begin{enumerate*}[label=\textup{(\arabic{*})}, mode=unboxed, resume]
\item\label{def:artinianity-and-the-like(3)}
An $\infty$-locally (resp., strongly $\infty$-locally) artinian premon will simply be called a \evid{locally} (resp., \evid{strongly locally}) artinian premon; and we shall say that the monoid $H$ is $\preceq$-artinian, [strongly] locally $\preceq$-artinian, etc., if the premon $(H, \preceq)$ is, resp., artinian, [strongly] locally artinian, etc.
\end{enumerate*}
\end{definition}

The notions of $\preceq$-artinianity and strong $\preceq$-artinianity (as per Definition \ref{def:artinianity-and-the-like}\ref{def:artinianity-and-the-like(2)}) are equivalent to the homonymous notions introduced in \cite[Definitions 3.8 and 3.11]{Tr20(c)} and further studied in \cite{Co-Tr-21(a), Co-Tr-22(a)}. The main novelty of the present work lies in the idea of turning $\preceq$-artinianity into an \emph{element-wise} condition, inspired by an online talk by F.~Gotti at the seminar of the Algebra and Number Theory research group of University of Graz in June 2022. 

\begin{remarks}
\label{rem:artinianity}
\begin{enumerate*}[label=\textup{(\arabic{*})}, mode=unboxed]
\item\label{rem:artinianity(1)}
A $1$-locally (resp., strongly $1$-locally) artinian premon is nothing else than an artinian (resp., strongly artinian) premon, and it is fairly obvious that, for all $h, k \in \mathbb N^+ \cup \{\infty\}$ with $h \le k$, an $h$-locally (resp., strongly $h$-locally) artinian premon is also $k$-locally (resp., strongly $k$-locally) artinian. In particular, artinian premons are locally artinian. The converse need not be true, as shown by Grams' counterexample in the basic case of the divisibility premon of a cancellative commutative monoid. 
\end{enumerate*}

\vskip 0.05cm

\begin{enumerate*}[label=\textup{(\arabic{*})}, mode=unboxed, resume]
\item\label{rem:artinianity(2)}
Fix $h, k \in \mathbb N_{\ge 2}$, and let $X = \{x_0, x_1, \ldots\}$ and $Y = \{y_1, y_2, \ldots\}$ be disjoint, countably in\-fi\-nite sets and $\sigma$ be a (strictly) increasing function on $\mathbb N$. Following \cite[Sect.~2.3]{Tr20(c)}, we denote by $H$ the quotient of the free monoid $\mathscr F(S)$ on $S := X \cup Y$ by the smallest monoid congruence $R^\#$ containing the set
\[
R := \bigcup_{r \in \mathbb N} \{(x_{rh}, x_{rh+1} \ast \cdots \ast x_{rh+h}), (x_{rh}, \underbrace{y_{\sigma(rk)+1} \ast y_{\sigma(rk)+2} \ast \cdots \ast y_{\sigma(rk+k)}}_{\sigma(rk+k) - \sigma(rk)  \text{ \scriptsize terms}})\} \subseteq \mathscr F(S) \times \mathscr F(S),
\]
where $\ast$ is the operation (of word concatenation) in $\mathscr F(S)$. 
Writing $\myoverline{\mathfrak u}$ for the congruence class modulo $R^\#$ of an $S$-word $\mathfrak u$, it is clear that $\myoverline{z}$ is a $\mid_H$-quark for every $z \in S \setminus \{x_0, x_h, x_{2h}, \ldots\}$, while the $\mid_H$-height of $\myoverlinesub{x_{rh}}$ is infinite for each $r \in \mathbb N$ (here we use that $h, k \ge 2$ and hence $\sigma(rk+k) - \sigma(rk) \ge 2$ by the hypothesis that $\sigma$ is increasing). It is then a routine exercise to check that (i) if $h < k$ and $\sigma$ is the identity map on $\mathbb N$, then the divisibility premon $(H, \mid_H)$ of $H$ is strongly $k$-locally artinian but not $h$-locally artinian, and (ii) if the growth rate of $\sigma$ is superlinear (e.g., $\sigma(n) := n^2$ for all $n \in \mathbb N$), then $(H, \mid_H)$ is strongly locally artinian but not $k'$-artinian for any $k' \in \mathbb N^+$ (we leave the details to the reader).
\end{enumerate*}

\vskip 0.05cm

\begin{enumerate*}[label=\textup{(\arabic{*})}, mode=unboxed, resume]
\item\label{rem:artinianity(3)}
Given a premon $\mathcal H = (H, \preceq)$ and an element $\bar{x} \in H$, we put $\downarrow_\mathcal{H} \bar{x} := \{x \in H \colon \allowbreak x \preceq \allowbreak \bar{x}\}$
and $\uparrow_\mathcal{H} \bar{x} := \allowbreak \{x \in \allowbreak H \colon \bar{x} \preceq x\}$. Similarly as in the case of a poset (see, e.g., \cite[p.~45]{Dav-Pri2002}), we call  $\downarrow_\mathcal{H} \bar{x}$ and $\uparrow_\mathcal{H} \bar{x}$, resp., the \evid{principal $\preceq$-ideal} and the \evid{principal $\preceq$-filter} generated by $\bar{x}$. Note that $\uparrow_\mathcal{H} \bar{x}$ is then a principal $\preceq^{\rm op}$-ideal and $\downarrow_\mathcal{H} \bar{x}$ is a principal $\preceq^{\rm op}$-filter, where $\preceq^{\rm op}$ is the dual (see, e.g., \cite[Example 3.3(1)]{Tr20(c)}) of the preorder $\preceq$.\\

\indent{}Now, it is evident that, for all $y, z \in H$, we have $y \preceq z$ if and only if $\downarrow_\mathcal{H} y \subseteq \allowbreak {\downarrow_\mathcal{H} z}$, if and only if $\uparrow_\mathcal{H} z \subseteq \allowbreak {\uparrow_\mathcal{H} y}$. It follows that an element $\bar{x} \in H$ is $\preceq$-artinian if and only if there is no sequence $x_1, x_2, \ldots$ in $H$ with $x_1 = \bar{x}$ and $\downarrow_\mathcal{H} x_{i+1} \subsetneq \allowbreak {\downarrow_\mathcal{H} x_i}$ (resp., $\uparrow_\mathcal{H} x_i \subsetneq {\uparrow_\mathcal{H} x_{i+1}}$) for all $i \in \mathbb N^+$. This allows for an ideal-theoretic interpretation of the notions of $\preceq$-artinianity and local $\preceq$-artinianity introduced in Definition \ref{def:artinianity-and-the-like}. Most notably, the principal $\mid_H$-filter generated by $\bar{x}$ is the principal two-sided ideal $H\bar{x}H$ of the monoid $H$; whence $H$ satisfies the ACCP if and only if it is $\mid_H$-artinian (cf.~\cite[Remark 3.9.4]{Tr20(c)}).
\end{enumerate*}
\end{remarks}

We are going to show that local artinianity is a sufficient condition for a premon $(H, \preceq)$ to be \evid{factorable} in the sense of \cite[Definition 3.2(4)]{Co-Tr-22(a)}, i.e., for each $\preceq$-non-unit to factor as a product of $\preceq$-ir\-red\-u\-ci\-bles (equivalently, we will say that the monoid $H$ is \evid{$\preceq$-factorable}). 

\begin{lemma}
\label{lem:locally-artinian-implies-locally-factorable}
Let $(H, \preceq)$ be a premon and $s$ be either an integer $\ge 2$ or $\infty$. Each locally $\preceq$-artinian $\preceq$-non-unit is then a product of $\preceq$-irreducibles of degree $s$.
\end{lemma}

\begin{proof}
Let $\Omega$ be the set of all $\preceq$-artinian $\preceq$-non-units that do not factor as a product of $\preceq$-irreducibles of degree $s$, and suppose for a contradiction that $\Omega$ is non-empty. It then follows from the well-foundedness of artinian preorders (see, e.g., Remark 3.9(3) in \cite{Tr20(c)}) that $\Omega$ has a $\preceq$-minimal element $\bar{x}$. In particular, $\bar{x}$ is neither a $\preceq$-unit nor a $\preceq$-irreducible (because the elements of $\Omega$ are neither $\preceq$-units nor products of $\preceq$-irreducibles of degree $s$). Therefore, $\bar{x} = yz$ for some $\preceq$-non-units $y, z \in H$ with $y \prec \bar{x}$ and $z \prec \bar{x}$, and at least one of $y$ and $z$ is not a product of $\preceq$-irreducibles of degree $s$ (or else so would be $\bar{x}$, which is absurd). But then either $y$ or $z$ is in $\Omega$ (note that $y$ and $z$ are $\preceq$-artinian elements of $H$, since $\bar{x}$ is $\preceq$-artinian and we have $y \prec \bar{x}$ and $z \prec \bar{x}$), contradicting that $\bar{x}$ is a $\preceq$-minimal element of the same set.
\end{proof}

The proof of the next result is now straightforward from Definition \ref{def:artinianity-and-the-like} and Lemma \ref{lem:locally-artinian-implies-locally-factorable}.

\begin{theorem}
\label{thm:FTF-with-local-artinianity}
If $(H, \preceq)$ is a locally artinian premon, then every $\preceq$-non-unit fac\-tors as a product of $\preceq$-ir\-red\-u\-ci\-bles of degree $s$ for all $s \in \mathbb N_{\ge 2} \cup \{\infty\}$ and, in particular, $H$ is a $\preceq$-factorable monoid.
\end{theorem}

In fact, Theorem \ref{thm:FTF-with-local-artinianity} is a refinement of \cite[Theorem 3.10]{Tr20(c)} and the existence part of \cite[Theorem 3.4]{Co-Tr-21(a)}, where the local artinianity of the premon $(H, \preceq)$ is replaced by the stronger condition of artinianity (and, incidentally, only $\preceq$-ir\-red\-u\-ci\-bles of \emph{finite} degree are being considered).

\begin{corollary}
\label{cor:characterization-of-factoriable}
Let $(H, \preceq)$ be a premon such that every $\preceq$-irreducible is a $\preceq$-quark or, more generally, has finite $\preceq$-height. Then $H$ is a $\preceq$-factorable monoid if and only if it is locally $\preceq$-artinian.
\end{corollary}

\begin{proof}
The ``if'' part is a consequence of Theorem \ref{thm:FTF-with-local-artinianity}. The ``only if'' part follows from considering that, if $H$ is a $\preceq$-factorable monoid and every $\preceq$-irreducible has finite $\preceq$-height, then the premon $(H, \preceq)$ is strongly artinian and hence locally artinian (as already noted in Remark \ref{rem:artinianity}\ref{rem:artinianity(1)}).
\end{proof}

In the light of Remark \ref{rem:artinianity}\ref{rem:artinianity(3)}, let us say that \evid{an element} $x$ in a monoid $H$ \evid{satisfies the \textup{ACCP}} if there is no sequence $x_1, x_2, \ldots$ in $H$ with $x_1 = x$ and $H x_i H \subsetneq Hx_{i+1} H$ for each $i \in \mathbb N^+$. Cohn's assertion that ``a commutative domain $R$ is atomic if and only if its multiplicative monoid $(R, \cdot\,)$ satisfies the ACCP'' then amounts to the statement that $(R, \cdot\,)$ is atomic if and only if \emph{every} element of $R$ satisfies the ACCP; and we are about to see that the truth is, in fact, not too far from Cohn's (false) claim.

\begin{corollary}
\label{cor:characterize-atomicity-in-acyclic-monoids}
An acyclic monoid is atomic if and only if it has a generating set each of whose elements satisfies the ACCP.
\end{corollary}

\begin{proof}
Let $H$ be an acyclic monoid. An element $x \in H$ is then a $\mid_H$-unit if and only if it is a unit. On the other hand, we gather from \cite[Corollary 4.4]{Tr20(c)} that $x$ is a $\mid_H$-irreducible if and only if it is an (ordinary) atom, if and only if it is a $\mid_H$-quark. It follows that $H$ is atomic if and only if it is $\mid_H$-factorable; and by Corollary \ref{cor:characterization-of-factoriable}, this is in turn equivalent to saying that $H$ is locally $\mid_H$-artinian. Hence every non-unit factors as a product of finitely many elements each of which satisfies the ACCP. Thus we are done, for it is obvious that units also satisfy the ACCP.
\end{proof}

First introduced in \cite[Definition 4.2]{Tr20(c)}, acyclic monoids abound in nature and provide an interesting alternative to cancellativity in the study of factorization in a non-commutative setting. Apart from unit-cancellative
commutative monoids, a number of non-commutative examples can be found in \cite[Example 5.4]{Co-Tr-22(a)}. In particular, we recall from the introduction that, for a monoid, being unit-cancellative and satisfying both the ACCPL and the ACCPR is equivalent to being acyclic and satisfying the ACCP.

\begin{examples}\label{exa:illustrations}
\begin{enumerate*}[label=\textup{(\arabic{*})}, mode=unboxed]
\item\label{exa:illustrations(1)}
Let $R$ be the subdomain of the univariate polynomial ring $\mathbb Q[X]$ over the rational field consisting of all polynomials whose constant term is an integer. We gather from the unnumbered example on p.~166 of \cite{Chap-2014} that $R$ is a non-atomic domain. This gives us a chance to illustrate how the sufficient condition in Corollary \ref{cor:characterize-atomicity-in-acyclic-monoids} can fail in practice.\\

\indent{}In fact, let $H$ be the multiplicative monoid of the non-zero elements of $R$ and suppose for a con\-tra\-dic\-tion that $H$ has a generating set $A$ each of whose elements satisfies the ACCP. Since $X$ is in $H$ and the only divisors of $X$ in $H$ are either integers or degree-one polynomials with zero constant term, it is clear that $qX \in A$ for some non-zero $q \in \mathbb Q$ (if the only generators in $A$ that divide $X$ were integers, then $X$ would not belong to the submonoid generated by $A$). However, $qX$ does not satisfy the ACCP (which is absurd), because $qX, \frac{1}{2} qX, \ldots, \frac{1}{2^i} qX, \ldots$ is a (strictly) $\mid_H$-decreasing sequence (note that $q_1 X \mid_H q_2 X$, for arbitrary $q_1, q_2 \in \mathbb Q$, if and only if $q_2 = q_1 k$ for some $k \in \mathbb Z$).
\end{enumerate*}

\vskip 0.05cm

\begin{enumerate*}[label=\textup{(\arabic{*})}, mode=unboxed, resume]
\item\label{exa:illustrations(2)} Let $r = a/b$ be a positive rational number smaller than $1$, with $a, b \in \mathbb N^+$, $a \ge 2$, and $\gcd(a,b) = 1$. We have already mentioned in Sect.~\ref{sec:intro} that the submonoid $H$ of $(\mathbb Q, +)$ generated by $1, r, r^2, \ldots$ is then an atomic monoid without the ACCP \cite[Corollary 4.4]{Cha-Got-Got21}. In fact, it is readily checked that, for all $i \in \mathbb N$,
\begin{equation*}
ar^{i+1} < ar^i = (b-a)r^{i+1} + ar^{i+1} \in ar^{i+1} + H, 
\end{equation*}
which shows that $a, ar, ar^2, \ldots$ is a (strictly) $\mid_H$-decreasing sequence. Since $H$ is a cancellative monoid, it follows that for an element $x \in H$ to be $\mid_H$-artinian (i.e., to satisfy the ACCP) it is necessary that $x \notin ar^i + H$ for each $i \in \mathbb N$. Interestingly, it turns out that the same condition is also sufficient.\\

\indent{}To see why, suppose $x \ne 0$ (or else there is nothing to prove) and denote by $\mathsf L_H(x)$ the set of all $n \in \allowbreak \mathbb N^+$ such that $x = a_1 + \cdots + a_n$ for some atoms $a_1, \ldots, a_n \in H$. We have that $\mathsf L_H(x) \ne \emptyset$ (note that the only unit of $H$ is the identity $0 \in \mathbb Q$), and \cite[Lemma 3.1(3)]{Cha-Got-Got19} yields that $|\mathsf L_H(x)| = \infty$ if and only if $x \in ar^i + H$ for some $i \in \mathbb N$. So, if $x \notin ar^i + H$ for every $i \in \mathbb N$, then $\mathsf L_H(x)$ is a non-empty finite subset of $\mathbb N^+$, which implies at once that the $\mid_H$-height of $x$ is finite and hence $x$ is $\mid_H$-artinian.
\end{enumerate*}
\end{examples}

\section*{Acknowledgments}
 
The paper was written during a visit at University of Graz, by invitation of Laura Cossu, in summer-fall 2022. The visit was funded through the Marie Sk\l{}odowska-Curie grant No.~101021791 from the European Union's Horizon 2020 research and innovation programme. I am indebted to Laura for her financial support, to Pedro A.~Garc\'ia-S\'anchez (University of Granada) for his help with Example \ref{exa:illustrations}\ref{exa:illustrations(2)}, to Daniel Smertnig (University of Graz) for his critical remarks, and to an anonymous referee for careful reading and many constructive comments.


\begin{thebibliography}{99}
%
\bibitem{Be-Br-Na-Sm22} J.\,P.~Bell, K.~Brown, Z.~Nazemian, and D.~Smertnig, \emph{On noncommutative bounded factorization domains and prime rings}, J.~Algebra \textbf{622} (May 2023), 404--449.
%
\bibitem{Boy-Coy2019} J.\,G.~Boynton and J.~Coykendall, \emph{An example of an atomic pullback without the ACCP}, J.~Pure Appl.~Algebra \textbf{223} (2019), 619--625. 
%
\bibitem{Chap-2014} S.\,T.~Chapman, \emph{A Tale of Two Monoids: A Friendly Introduction to Nonunique Factorizations}, Math.~Mag.~\textbf{87} (2014), No.~3, 163--173.
%
\bibitem{Cha-Got-Got19} S.\,T.~Chapman, F.~Gotti, and M.~Gotti, \emph{Factorization invariants of Puiseux monoids generated by geometric sequences}, Comm.~Algebra \textbf{48} (2020), No.~1, 380--396.
%
\bibitem{Cha-Got-Got21} S.\,T.~Chapman, F.~Gotti, and M.~Gotti, \emph{When Is a Puiseux Monoid Atomic?}, Amer.~Math.~Monthly \textbf{128} (2021), No.~4, 302--321.
%
\bibitem{Co64} P.\,M.~Cohn, \emph{Free ideal rings}, J.~Algebra \textbf{1} (1964), 47--69.
%
\bibitem{Co69} P.\,M.~Cohn, \emph{Torsion modules over free ideal rings}, Proc.~London Math.~Soc.~III.~Ser.~\textbf{17} (1967), 577--599.
%
\bibitem{Co68} P.\,M.~Cohn, \emph{Bezout rings and their subrings}, Math.~Proc.~Cambridge Phil.~Soc.~\textbf{64} (1968), No.~2, 251--264.
%
\bibitem{Co73} P.\,M.~Cohn, \emph{Unique Factorization Domains}, Amer.~Math.~Monthly \textbf{80} (1973), No.~1, 1--18.
%
\bibitem{Co06} P.\,M.~Cohn, \emph{Free Ideal Rings and Localization in General Rings}, New Math.~Monogr.~\textbf{3}, Cambridge Univ.~Press, 2006.
%
\bibitem{Co-Tr-21(a)} L.~Cossu and S.~Tringali, \emph{Abstract Factorization Theorems with Applications to Idempotent Factorizations}, to appear in Israel J.~Math.~(\url{https://arxiv.org/abs/2108.12379}). 
%
\bibitem{Co-Tr-22(a)} L.~Cossu and S.~Tringali, \emph{Factorization under Local Finiteness Conditions}, to appear in J.~Algebra (\url{https://arxiv.org/abs/2208.05869}). 
%
\bibitem{Coy-Go2019} J.~Coykendall and F.~Gotti, \emph{On the atomicity of monoid algebras}, J.~Algebra \textbf{539} (2019), 138--151. 
%
\bibitem{Dav-Pri2002} B.\,A.~Davey and H.\,A.~Priestley, \emph{Introduction to Lattices and Order}, Cambridge Univ.~Press, 2002 (2nd ed.).
%
%
\bibitem{Fa-Tr18} Y.~Fan and S.~Tringali, \emph{Power monoids: A bridge between Factorization Theory and Arithmetic Combinatorics}, J.~Algebra \textbf{512} (Oct.~2018), 252--294.
%
\bibitem{GeHK06} A.~Geroldinger and F.~Halter-Koch, \emph{Non-Unique Factorizations. Algebraic, Combinatorial and Analytic Theory}, Pure Appl.~Math.~\textbf{278}, Chapman \& Hall/CRC, Boca Raton (FL), 2006.
%
\bibitem{Ger-Zho2020} A.~Geroldinger and Q.~Zhong, \emph{Factorization theory in commutative monoids}, Semigroup Forum \textbf{100} (2020), 22--51.
%
\bibitem{Gil84} R.~Gilmer, \emph{Commutative Semigroup Rings}, Chicago Lect.~Math., Univ.~of Chicago
Press, Chicago, IL, 1984. 
%
\bibitem{GotLi22(a)} F.~Gotti and B.~Li, \emph{Atomic semigroup rings and the ascending chain condition on principal ideals}, to appear in Proc.~Amer.~Math.~Soc.~(\url{https://arxiv.org/abs/2111.00170}).
%
%
\bibitem{Grams74} A.~Grams, \emph{Atomic rings and the ascending chain condition for principal ideals}, Math.~Proc.~Cambridge Phil.~Soc.~\textbf{75} (1974), No.~3, 321--329.
%
\bibitem{Ho95} J.\,M.~Howie, \emph{Fundamentals of Semigroup Theory}, London Math.~Soc.~Monogr.~Ser.~\textbf{12}, Oxford Univ.~Press, 1995.
%
\bibitem{Roit93} M.~Roitman, \emph{Polynomial extensions of atomic domains}, J.~Pure Appl.~Algebra \textbf{87} (1993), No.~2, 187--199.
%
\bibitem{Tr20(c)} S.~Tringali, \emph{An abstract factorization theorem and some applications}, J.~Algebra \textbf{602} (July 2022), 352--380.
%
\bibitem{Zaks82} A.~Zaks, \emph{Atomic rings without a.c.c.~on principal ideals}, J.~Algebra \textbf{80} (1982), 223--231.
%
\end{thebibliography}
\end{document}